\documentclass[12pt]{article}

\usepackage[usenames]{color}
\usepackage{theorem,amssymb,amsmath}
\usepackage{graphicx}

\usepackage[colorlinks=true,
linkcolor=webgreen,
filecolor=webbrown,
citecolor=webgreen]{hyperref}

\definecolor{webgreen}{rgb}{0,.5,0}
\definecolor{webbrown}{rgb}{.6,0,0}

\advance \topmargin by -\headheight
\advance \topmargin by -\headsep
\evensidemargin \oddsidemargin
\newcommand{\seqnum}[1]{\href{https://oeis.org/#1}{\rm \underline{#1}}}

\author{J.-P. Allouche \\
CNRS, IMJ-PRG \\
Sorbonne, 4 Place Jussieu \\
F-75252 Paris Cedex 05 \\
France \\
{\tt jean-paul.allouche@imj-prg.fr}\\
\and
J. Shallit \\
School of Computer Science \\
University of Waterloo \\
Waterloo, Ontario  N2L 3G1 \\
Canada \\
{\tt shallit@uwaterloo.ca}
\and
R. Yassawi \\
School of Mathematics and Statistics \\
The Open University \\
Walton Hall, Kents Hill\\
Milton Keynes MK76AA \\
United Kingdom \\
{\tt reem.yassawi@open.ac.uk}
}

\title{How to prove that a sequence is not automatic}

\date{ }

\def \proof{\bigbreak\noindent{\it Proof.\ \ }}

\def \endpf{{\ \ $\Box$ \medbreak}}
\def\suchthat{ \, : \, }
\def\Enn{\mathbb{N}}

\newtheorem{theorem}{Theorem}
\newtheorem{lemma}[theorem]{Lemma}

\newtheorem{corollary}[theorem]{Corollary}

\theorembodyfont{\rm}

\newtheorem{remark}[theorem]{Remark}
\newtheorem{example}[theorem]{Example}
\newtheorem{definition}[theorem]{Definition}

\begin{document}

\maketitle

\begin{abstract}
Automatic sequences have many properties that other sequences 
(in particular, non-uniformly morphic sequences) do not 
necessarily share. In this paper we survey a number of 
different methods that can be used to prove that a given 
sequence is not automatic. When the sequences take their 
values in the finite field ${\mathbb F}_q$, this also permits 
proving that the associated formal power series are 
transcendental over ${\mathbb F}_q(X)$.
\end{abstract}

\section{Introduction}

Automatic sequences can be found in several fields, 
particularly in view of their nature being ``deterministic 
but possibly chaotic-like''. They are more ``regular'' than 
other sequences; in particular, than general non-uniformly 
morphic sequences. Several papers prove that given sequences 
or families of sequences are not automatic, using a variety 
of methods. The purpose of this survey is to give a manual for proving (or trying to prove) that a given sequence is not automatic. 
The method essentially consists of finding, for each 
considered sequence, a relatively ``easy-to-check'' criterion 
for being automatic that is not satisfied by the sequence. 
In the case where the sequence takes its values in a finite 
field ${\mathbb F}_q$, proving that it is not $q$-automatic 
gives a proof of the transcendence of the associated formal 
power series over the field ${\mathbb F}_q(X)$, by means of a celebrated theorem 
of Christol (see \cite{Christol, CKMFR}).

\medskip
For a general approach to automatic and morphic sequences, 
the reader can consult, e.g., \cite{AS, Fogg, Haeseler, Queffelec}. 
We recall some definitions here. 

\begin{itemize}

\item If $A$ is an {\em alphabet} (i.e., a finite set), 
we let $A^*$ denote the set of {\em words} over $A$, 
including the empty word (i.e., the set of finite sequences 
on $A$, including the empty sequence). 

\item The set $A^*$ can be equipped with a structure of a 
(free) {\em monoid}, with multiplication being {\em concatenation} 
of words. 

\item A {\em morphism} $\varphi$ from an alphabet $A$ to an 
alphabet $B$ is a map $A^* \to B^*$ such that for all
$u, v \in A^*$ one has $\varphi(uv) = \varphi(u) \varphi(v)$.  
Clearly, a morphism is completely specified from its values on 
$A$ alone. If $A = B$, the morphism is called a {\em morphism 
on $A$}. If $A = \{a_1, \ldots, a_d\}$, the {\em transition matrix} 
(or {\em adjacency matrix}) of $\varphi$ is the matrix 
$M = (m_{i,j})$ where $m_{i,j}$ is the number of occurrences 
of the letter $a_i$ in $\varphi(a_j)$.

\item The morphism $\varphi$ is called {\em uniform} if the 
lengths of the images  of each letter in $A$ by $\varphi$ are 
the same. If this length is equal to $q$, the morphism is 
called {\em $q$-uniform} or a {\em $q$-morphism}. 

\item Let $\varphi$ be a morphism on the alphabet $A$. 
If there exist a letter $a \in A$ and a non-empty word 
$v \in A^*$ such that $\varphi(a) = a v$ and no $\varphi^k(a)$ 
is empty, the sequence of words $a, \varphi(a), \varphi^2(a), 
\ldots$ converges (for the product topology) to an infinite 
sequence on $A$, namely the sequence
$$
a \, v \, \varphi(v) \, \varphi^2(v)\cdots \varphi^k(v) \cdots,
$$
which is a fixed point of $\varphi$ extended to infinite 
sequences on $A$. This limit is called a {\em fixed point} 
or an {\em iterative fixed point} of the morphism $\varphi$.

\item If a sequence is the image of the fixed point of a 
morphism by a $1$-morphism, it is called {\em morphic}.

\item If a sequence is morphic for a $q$-uniform morphism, it is called {\em $q$-automatic}. A sequence that is
$q$-automatic for some $q \geq 2$ is called {\em automatic}.

\item For $x$ a finite or infinite word
by $x[i..j]$ we mean $x[i]\cdots x[j]$, where
$x[i]$ is the $i$th letter of $x$.

\end{itemize}

\bigskip

The most famous example of an automatic sequence (more specifically, a
$2$-automatic sequence) is the (Prouhet-)Thue-Morse sequence 
${\mathbf u} = (u_n)_{n \geq 0}$ where $u_n$ is the sum, 
reduced modulo $2$, of the binary digits of $n$. It is not 
difficult to see that this sequence is the iterative fixed 
point, starting with $0$, of the morphism defined on $\{0, 1\}$ 
by $0 \to 01$, $1 \to 10$.

\medskip

One of the most famous examples of a morphic sequence is the 
binary Fibonacci sequence ${\bf f} = 01001010\cdots$, defined as the iterative fixed point,
starting with $0$, of the morphism defined on $\{0, 1\}$ by 
$0 \to 01$, $1 \to 0$.

\section{Infinite $q$-kernels}

The $q$-kernel of a sequence
${\bf a}=(a_n)_{n \geq 0}$ is the set of linearly-indexed subsequences 
$${\rm Ker}_q \ {\bf a} = \{(a_{q^k n + r})_{n \geq 0}, 
\ k \geq 0, \  r \in [0, q^k - 1]\}.$$

A necessary and sufficient condition for a sequence 
to be $q$-automatic is 
that its $q$-kernel be finite.
Thus, to prove that a sequence is not $q$-automatic, it 
suffices to exhibit some subset of its $q$-kernel of infinite cardinality. 
In particular, the following result often proves useful.
\begin{theorem}\label{kernel}
Let ${\mathbf a} = (a_n)_{n \geq 0}$ be a sequence and $q$ 
be an integer $\geq 2$. If there exists a sequence of integers 
$(r_k)_{k \geq 0}$ such that $r_k \in [0,q^k)$ and the subsequences
$(a_{q^k n + r_k})_{n \geq 0}$ are all distinct, then the 
sequence ${\mathbf a}$ is not $q$-automatic.
\end{theorem}

\proof Immediate from the finite kernel property for automatic 
sequences.  \endpf

\begin{example}\label{ex-kernel1}

This last result has been used several times in the literature.
In particular, a theorem of Christol \cite{Christol, CKMFR} 
asserts that a formal power series $\sum a_n X^n$ with 
coefficients in the finite field ${\mathbb F}_q$ is 
transcendental over the field of rational functions 
${\mathbb F}_q(X)$ if and only if the sequence $(a_n)$ is 
$q$-automatic. Hence, proving that the series $\sum a_n X^n$ 
is transcendental over ${\mathbb F}_q(X)$ is equivalent to 
proving that the $q$-kernel of the sequence $(a_n)$ is not 
finite. Here are some examples of results on the non-finiteness
of $q$-kernels of sequences.

\begin{itemize}

\item A variation on this method was used 
\cite[Lemme fondamental, p.\ 281]{allouche82} to prove that 
the sequence of the $p$-ary sum of digits of $R(n)$, reduced 
modulo $p$, is not $p$-automatic (where $R$ is a polynomial 
of degree at least $2$ that sends the integers to the 
integers). Also see the generalization \cite{allouche-salon}, 
where the sequence of $p$-ary sums of digits is replaced with 
any {\em quasi-strongly-$B$-additive} sequence.

\item (Non-)finiteness of the kernel is used to prove 
transcendence/algebraicity results for formal Drinfeld 
modules in \cite{Cadic1, Cadic2}, via an unpublished proof 
by the first author of a conjecture of Laubie 
(see \cite[Proposition~3.3.1]{Cadic1} or
\cite[Proposition~1]{Cadic2}): 
Let $(a_n)_{n \geq 0}$ be a sequence with values in the
finite field ${\mathbb F}_q$. Then the formal power series
$\sum a_n X^{q^n}$ is algebraic over ${\mathbb F}_q(X)$ if
and only if the formal power series $\sum a_n X^n$ is rational
(i.e., if and only if the sequence $(a_n)_{n \geq 0}$ is
eventually periodic).

\item The infinite fixed point ${\bf v} = (v_n)_{n \geq 0}$ of the morphism $a \to aab$, 
$b \to b$ is not $2$-automatic, as proved in 
\cite{all-shall-betrema}. The study of the sequences 
$(v_{2^k n + 2^k - k})_{n \geq 0}$ for $k \geq 1$ reveals that 
they are all distinct. This fixed point occurs in Von Neumann's 
recursive definition of the integers (see, e.g., 
\cite[Section~3.2]{Goldrei}).

\item 
Let $p$ be a prime number and $q$ a power of $p$. 
Let ${\mathbb F}_q$ denote the finite field with $q$ elements.
The formal power series $\Pi_q$ is an element of 
${\mathbb F}_q((X^{-1}))$ and is an analog of $\pi$. 
A proof that $\Pi_q$ is transcendental over ${\mathbb F}_q(x)$ 
uses Christol's theorem and the fact that $(b_n)_{n \geq 0}$ 
is not $q$-automatic, where $(b_n)_{n \geq 0}$ is the 
characteristic function of the integers that can be written 
as a finite sum $n = \sum (q^j-1)$. It can be proved that 
the sequences $(b_{q^k n + q^k - k})_{n \geq 0}$ are all 
distinct \cite{allouche90}.

\item 
The transcendence of the series $\Pi_q$ in the previous item, 
as well as the transcendence of the so-called {\em bracket 
series}, were proved in \cite{allouche-lotharingien} by 
showing
that a certain sequence $(v_n)_{n \geq 0}$ is not 
$q$-automatic, thanks to a corollary of Theorem~\ref{kernel} 
above: {\it if a sequence is $q$-automatic, then there are 
finitely many sequences distinct of the form $(v_{q^k (n+1) - 1})_{n \geq 0}$, 
and hence the sequence $(v_{q^k -1})_{k \geq 0}$ is ultimately 
periodic.} A result of the same kind in the case of 
$2D$-automatic sequences was used 
in \cite{AHPPS}.

\item Other transcendence results for values of Carlitz 
functions analogous to the Riemann zeta function, to the 
logarithm, etc. have been obtained in approaches similar 
to the previous two items, see, e.g.,
\cite{Berthe93, Berthe94, Berthe95, allouche-gamma, Mendes-Yao, Yao2002, Wen-Yao, Hu2018}.

\item A very sophisticated criterion of transcendence based on
the non-finiteness of the $q$-kernel of a sequence is given in
\cite{Yao}.

\item The fixed point $(d_n)_{n \geq 0}$ of the morphism 
$1 \to 121$, $2 \to 12221$ is not $2$-automatic, which results 
 from the (non-trivial) property that the subsequences 
$(d_{2^{2k} n})_{n \geq 0}$ are all distinct (see \cite{AAS}, 
where the morphism above occurs in the drawing of a classical 
{\em kolam}). Note that the sequence $(e_n)_{n \geq 0}$ defined
by $e_n = d_{n+1}$ for all $n \geq 0$ is the fixed point of 
the morphism $2 \to 211$, $1 \to 2$.

\item Generalizing the result for the morphism $2 \to 211$, 
$1 \to 2$ in the previous item, it can be proved similarly 
that, for $k \geq 1$, the fixed point of the morphism 
$1 \to 1^{k-1}2$, $2 \to 1^{k-1}21^{k+1}$ is not 
$(k+1)$-automatic (see \cite{ASWWZ} where this family of 
morphisms occurs in the study of certain sum-free sets).

\item It is proved in \cite{MRSS} that, if a sequence has 
arbitrarily long blocks in common with a Sturmian sequence, 
then it cannot be $q$-automatic for any $q \geq 2$. This is 
generalized from Sturmian sequences to generalized polynomials 
in \cite{BK1}.

\item It is proved in \cite{Spiegelhofer} that the sequence of gaps 
between consecutive occurrences of a block $w$ with $|w| > 1$ in the
Thue-Morse sequence is substitutive but not $k$-automatic for
any $k \geq 2$.
\end{itemize}

\end{example}

Other kinds of proofs of non-finiteness of $q$-kernels 
have been 
used for ``arithmetic'' sequences, in particular for 
{\em multiplicative sequences}. Recall that a sequence 
$(a_n)_{n \geq 1}$ is called {\em multiplicative} if, for all
$m, n \geq 1$ such that $\gcd(m, n) = 1$, one has 
$f(mn) = f(m)f(n)$. Also recall that the sequence 
$(a_n)_{n \geq 1}$ is called {\em completely multiplicative} 
if, for all $m, n \geq 1$, one has $f(mn) = f(m)f(n)$.

\bigskip
 
The following theorem is \cite[Theorem~2]{Yazdani}.
\begin{theorem}[{\rm \cite{Yazdani}}]\label{yazdani1}
Let $v > 1$ be an integer and $f$ a multiplicative function. 
Assume that for some integer $h \geq 1$ there exist infinitely 
many primes $q_1$ such that $f(q_1^h) \equiv  0 \pmod v$. 
Furthermore assume that there exist relatively prime integers 
$b$ and $c$ such that for all primes $q_2 \equiv c \pmod b$ we
have $f(q_2) \not\equiv 0 \pmod v$. Then the sequence
$(f(n))_{n \geq 1} \pmod v$ is not $q$-automatic for any 
$q \geq 2$.
\end{theorem}   
    
\begin{example}\label{ex-kernel2}

There are several papers about (completely) multiplicative 
functions and their (non-)automaticity in the literature. 
Here we give one theorem and a few recent references.

\begin{itemize} 

\item Recall that the multiplicative sequences $\sigma_m$,
$\varphi$ and $\mu$, are defined as follows: 
for $m \geq 1$ $\sigma_m(n) := \sum_{d|n} d^m$, $\varphi(n)$ 
is the Euler totient function, and $\mu(n)$ is the M\"obius 
function. Theorem~\ref{yazdani1} was used in \cite{Yazdani} 
to prove that $(\sigma_m(n) \pmod v)_{n \geq 1}$ and 
$(\varphi(n) \pmod v)_{n \geq 1}$ are not $q$-automatic for
any $q \geq 2$ and $v \geq 3$.  The case $v = 2$ is also addressed
in \cite[Thm.~7]{Yazdani}, where the author uses a theorem of
Minsky and Papert (Theorem~\ref{minsky-papert} below), and 
the fact that $(\mu(n) \pmod v)_{n \geq 1}$ is not $q$-automatic for any 
$q \geq 2$ and $v \geq 2$.

\item Recent papers on (completely) multiplicative sequences that
are (non-)automatic use a variety of methods: see
\cite{allouche-goldmakher, Hu, Li, KK1, KK2, Konieczny, KLM}.
In particular Konieczny, Lema\'nczyk, and M\"uller \cite{KLM} 
give a complete characterization of automatic multiplicative
sequences.

\end{itemize}

\end{example}

\begin{remark}
A classical property of regular languages is that they 
satisfy the ``pumping lemma''\footnote{Though some people 
try to translate it literally, the correct name of this lemma 
in French is ``le lemme de l'\'etoile".} (see, e.g., 
\cite[Lemma~4.2.1]{AS}).    Let $C_q$ be the set of all canonical base-$q$ representations of natural numbers (with no leading zeroes).   Since a sequence $(a_n)_{n \geq 0}$ 
is $q$-automatic if and only if all the languages 
$$
L_b = \{n_q, \ \text{$n_q \in C_q$ is the 
$q$-ary expansion of $n$ and} \ a_n = b\}
$$
are regular, a way to prove that a sequence $(a_n)_{n \geq 0}$ 
is not $q$-automatic is to prove that for some value $b$, the 
language $L_b$ is not regular ---which can be done, e.g., by
using the pumping lemma. 

As an example of this approach, consider the characteristic 
sequence of the set $S = \{ 2^n (2^n -1 ) , n \geq 0 \}$. 
Here the language of base-$2$ expansions of $S$ is 
$\{ 1^n 0^n , n \geq 0 \}$, a classical non-regular language. 
Thus the characteristic sequence of $S$ is not $2$-automatic. 
Another example is given in \cite{Shallit-automaticity4}: the 
fixed point beginning with $c$ of the morphism  $c \to cba$, 
$a \to aa$, and $b \to b$, which is proved to be 
non-$2$-automatic by showing that the language
$L := \{1 0^{n - \lfloor \log_2 n \rfloor - 1}(n)_2, 
\ n \geq 1\} \cup \{1\}$ is not regular.

\end{remark}

\begin{remark}
We have not explicitly spoken of $q$-automata. But, actually,
using the finiteness of the $q$-kernel or the pumping lemma is 
essentially possible because the ``$q$-automaton'' behind an
automatic sequence has a {\em finite} number of states. A nice  
recent paper \cite{Lipka} invokes this property to prove that 
$(\ell_b(n))_{n \geq 1}$, is not a $q$-automatic sequence for 
any $q \geq 2$, where $\ell_b(n)$ is the last nonzero digit of 
$n!$ and $b \geq 2$ is a fixed base such that, if
$b = p_1^{a_1} p_2^{a_2}\cdots$, then there exist at least two $p_i$'s 
such that $a_i(p_i-1) = \max\{a_j(p_j-1), \ j = 1, 2,\ldots \}$. These 
integers form the sequence \seqnum{A135710} in \cite{OEIS} 
(\seqnum{A135710} = $12, 45, 80, 90, 144, 180, 189, 240, 360, \ldots$).
\end{remark}

\begin{remark}   Similarly, the Myhill-Nerode theorem (see, e.g., \cite[Thm.~3.9]{Hopcroft}) can be used to prove that a language is not regular.  In
fact, the $k$-kernel of a binary sequence $(s_n)_{n \geq0}$  is essentially the same size as the number of Myhill-Nerode equivalence classes of the language of reversed canonical base-$k$ representations of $\{n , s_n = 1 \}$.
\end{remark}

\section{Irrational frequencies}

Suppose a sequence $(a_n)_{n \geq 0}$ is $q$-automatic for 
some $q \geq 2$. Then, if the {\em letter frequencies} 
$$
\displaystyle \lim_{n \to \infty} \frac{1}{N} \#\{n \leq N, 
\, a_n = b\}
$$
exist, they must be rational. Hence if the frequency of 
occurrence of some letter in a sequence $(a_n)_{n\geq 0}$ 
taking its values in a finite alphabet exists and is 
irrational, the sequence cannot be $q$-automatic for any 
$q \geq 2$.

\begin{example}
An example of the claim above is given by the Sturmian 
sequences (recall that a sequence is Sturmian if, for all 
$k \geq 0$, it contains exactly $k+1$ distinct
blocks of length $k$). 
A Sturmian sequence necessary takes only two values, and 
the frequencies of occurrence of these two values are 
irrational. Note that this remark makes the title of 
\cite{Tapsoba} redundant, and underlines a mistake in 
\cite{Brimkov-Barneva} (about that paper, also see the review 
MR 2006a:68140 by P. S\'e\'ebold). Speaking of non-automaticity
of Sturmian sequences, it is worth noting that, more generally,
{\em dendric sequences} cannot be $q$-automatic for any 
$q \geq 2$: this is proved in \cite[Corollary~15]{BDDLP}, by 
first proving that these sequences admit no rational 
topological dynamical eigenvalue (see Section~\ref{dyn}). 
Note that, as indicated in that paper, Sturmian sequences, 
episturmian sequences, and codings of interval exchanges are
particular examples of dendric sequences.
\end{example}

It may happen that all frequencies exist and are rational, 
but, in this case, a variation on the observation above can be 
useful:

\begin{theorem}\label{irrational1}
If a sequence $(a_n)_{n \geq 0}$ with values in a finite 
alphabet is such that the frequency of some block occurring 
in this sequence exists and is irrational, then the sequence 
cannot be $q$-automatic for any $q \geq 2$.
\end{theorem}

\proof The sequence of consecutive overlapping blocks of 
length $k$ occurring in a $q$-automatic sequence is also 
$q$-automatic. 
\endpf

\begin{example}
Theorem~\ref{irrational1} was used in the following examples.
\begin{itemize}

\item
To show that the sequence of moves in the cyclic tower of 
Hanoi algorithm is not $q$-automatic, for any $q \geq 2$, 
it was proved in \cite{allouche-hanoi} that the frequency 
of some three-letter word on the alphabet of moves exists 
and is not rational.

\item
It is shown in \cite{allouche-palanga} that the language of 
all primitive words over a finite alphabet is not unambiguously 
context-free by proving that the square of the M\"obius 
function, $(\mu^2(n))_{n \geq 1}$, is not automatic: the 
frequencies of the values taken by $\mu^2$ are $6 / \pi^2$ 
and $1 - 6/\pi^2$, which are irrational. 

\end{itemize}

\end{example}

\begin{remark}

\ { }

\begin{itemize}

\item Unfortunately, this method does not work all the time: 
there exist morphic sequences that are non-automatic, ``although''
the frequencies of all words occurring in the sequence exist 
and are rational. Namely consider the fixed point of the 
morphism $2 \to 211$, $1 \to 2$. The dominant eigenvalue of 
the matrix of this morphism is equal to $2$. Furthermore the 
morphism is {\it primitive}:  this means that there is a constant $d$ such that $a$ appears in $f^d (b)$ for all letters
$a,b$ in the alphabet. For each $\ell \geq 1$ the associated 
morphism generating the sequence of overlapping blocks of 
length $\ell$ is also primitive and has $2$ as dominant 
eigenvalue (see, e.g.,  \cite[Section~5.4.1]{Queffelec}). 
Hence the frequency of occurrence of every block is a rational 
number (recall that the vector of frequencies of a fixed point 
of a primitive morphism is the normalized eigenvector 
associated with the dominant eigenvalue, and that linear 
equations with rational coefficients have rational 
solutions). But the fixed point of $2 \to 211$, $1 \to 2$ is 
not $q$-automatic for any $q \geq 2$ (see \cite{AAS}).

\item If the frequency for a given letter occurring in a 
sequence does not exist, one can replace the limit in the 
definition of the frequency with limsup or liminf: 
a result in \cite{Bell} asserts that, if the sequence is 
automatic, then both quantities limsup and liminf are rational.

\end{itemize}

\end{remark}

Another result dealing with frequencies, stated in
\cite{ADQ}, is worth noting.

\begin{theorem}[{\rm \cite{ADQ}}]\label{irrational2}
If the adjacency matrix of a primitive non-uniform morphism 
has an irrational dominant eigenvalue, then an iterative fixed 
point of this morphism cannot be automatic.
\end{theorem}

\begin{example}
Theorem~\ref{irrational2} was used in \cite{ADQ} to prove 
the non-automaticity of fixed points of morphisms related to 
Grigorchuk-like groups. For example, the fixed point of the 
morphism defined on $\{a, b, c, d\}$ by $a \to aca$, $b \to d$,
$c \to aba$, $d \to c$ (see \cite[Theorem~4.1]{bs}) is not 
automatic. Namely, as noted in \cite{ADQ}, the matrix of this 
morphism is primitive and its characteristic polynomial, which 
is equal to $x^4 - 2 x^3 - 2 x^2 - x + 2$, clearly has no 
rational root.
\end{example}

\section{Synchronization}

Suppose $f$ is a function from $\Enn$ to $\Enn$.  If there is a deterministic finite automaton that recognizes, in
parallel, the base-$k$ representations of $n$ and $f(n)$,
then we say that $f$ is \textit{$k$-synchronized}.
The following result is very useful for proving
sequences not automatic.
\begin{theorem}
If $f$
is $k$-synchronized, then
\begin{itemize}
    \item[(a)] $f = \mathcal{O}(n)$;
    \item[(b)] If $f=o(n)$ then $f(n) = \mathcal{O}(1)$.
    \item[(c)] If there is an increasing subsequence
    $n_1 < n_2<\cdots$ such that
    $lim_{i \rightarrow\infty} f(n_i)/n_i = 0$,
    then there is a constant $C$ such that
    $f(n) = C$ for infinitely many $n$.
\end{itemize}
\label{synch1}
\end{theorem}
For proofs, see \cite{Carpi-Maggi,Shallit:2021}.

Many functions dealing with $k$-automatic sequences are
$k$-synchronized.   A fairly detailed list is contained in
\cite{Shallit:2021} and includes such quantities as
\begin{itemize}
    \item appearance (length of shortest prefix containing all length-$n$ blocks); \cite{Charlier&Rampersad&Shallit:2012}
    \item repetitivity index (minimum distance between two consecutive occurrences of a length-$n$ block); \cite{Carpi&DAlonzo:2009}
    \item the uniform recurrence function (maximum distance between two consecutive occurrences of a length-$n$ block); \cite{Charlier&Rampersad&Shallit:2012}
    \item condensation (length of the shortest block containing all length-$n$ blocks); \cite{Goc&Henshall&Shallit:2013}
    \item separator length (length of the shortest word beginning at position $n$ not appearing previously in the sequence); \cite{Garel:1997,Carpi-Maggi}
    \item palindrome separation (longest distance between two consecutive length-$n$ blocks, both of which are palindromes); \cite{Shallit:2021}
    \item repetition word length (for each $n$, the length of the shortest prefix $w$ of ${\bf x}[n..\infty]$ for which either $w$ is a suffix of ${\bf x}[0..n-1]$ or vice versa); \cite{Mignosi&Restivo:2012}
    \item largest square centered at a given position; \cite{Shallit:2021}
    \item shortest square beginning at a given position; \cite{Shallit:2021}
    \item longest palindromic suffix of a length-$n$ prefix; \cite{BlondinMasse&Brlek&Garon&Labbe:2008}
    \item the Bugeaud-Kim function (the length of the shortest prefix of $\bf x$ containing two possibly overlapping occurrences of some length-$n$ block);
    \cite{Bugeaud&Kim:2018}
    \item block complexity; \cite{Goc} and
    \item first occurrence of a run of length $\geq n$.  \cite{Shallit:2021}
\end{itemize}
As a consequence, we immediately get that if $\bf x$ is a $k$-automatic sequence, then the bounds in Theorem~\ref{synch1} hold.   This gives a method for proving some sequences non-automatic.   Let us consider some examples.

\begin{example}
A {\it run} in a sequence is a block of consecutive identical values.
Schlage-Puchta proved the following lemma about runs in automatic 
sequences:  if an automatic sequence $(a_n)_{n \geq 0}$ has 
arbitrarily long runs, then there exists a constant $c > 0$ such 
that $a_n = a$ for $n \in [x, (1+c)x]$ and infinitely many $x$.  
See \cite{Schlage-Puchta}. With this lemma he was able to prove that
the sequence $(\mu(n) \bmod p)_{n \geq 1}$ is not 
automatic, for all primes $p$.   Schlage-Puchta's lemma immediately follows  from Theorem~\ref{synch1} and the 
observation that the function $f$ mapping $n$ to the first position 
$m$ where $a_m = a_{m+1} = \cdots = a_{m+n-1}$ is $k$-synchronized.  
\end{example}

Even further, the $O(n)$ upper bound on the growth rate of automatic sequences applies to sequences defined over many other kinds of numeration
systems, such as Fibonacci numeration, Tribonacci numeration, and so forth.   

\begin{example}
Consider the fixed point $\bf v$ of the morphism
$a \to aab$, $b \to b$.   The function mapping $n$ to the starting position
of the first occurrence of a run of length $n$ in a $k$-automatic sequence is
$k$-synchronized and hence must be in $\mathcal{O}(n)$.  But the the first occurrence of run of length $n$ in
$\bf v$ begins at position $2^{n+1}-n-1$, which is clearly not in $\mathcal{O}(n)$.
Hence $\bf v$ is not $k$-automatic for any $k$.  By the remark above, $\bf v$
cannot be `automatic' in any numeration system at all (e.g., 
Fibonacci, Ostrowski, etc.), provided the numeration system has certain
properties.
\end{example}

\section{Block complexity}

As we saw in the last section, the (block-)complexity (aka factor complexity, 
aka subword complexity) of a sequence ${\mathbf a}$ is the 
function 
$p_{\mathbf a}(n)$ counting the number of distinct blocks of 
length $n$ that occur in ${\mathbf a}$. Hence the block complexity 
of an automatic sequence is in ${\mathcal O}(n)$ (see, e.g.,
\cite[Thm.~2]{Cobham72}). Thus we have
\begin{theorem}\label{complexity}
If the (block-)complexity of a sequence taking its values in a 
finite alphabet is not in ${\mathcal O}(n)$, then the sequence 
cannot be automatic. If the appearance function of a sequence
taking its values in a finite alphabet is not in 
${\mathcal O}(n)$, then the sequence cannot be automatic.
\end{theorem}
This is one of the most useful criteria for proving
non-automaticity.

\begin{example}\label{ex-complexity}
Theorem~\ref{complexity} was used in various contexts.

\begin{itemize}

\item A first example, in the more general context of 
$2D$-sequences, is the study of complexity of the Pascal 
triangle modulo $d$: in this case the complexity is the 
rectangle-complexity; $p(u,v)$ is the number of different 
rectangles of size ($u \times v$). It was proved in 
\cite{allouche-berthe} that, for $d \geq 2$, the complexity 
of the sequence $({m \choose n} \pmod d)_{m, n \geq 0}$ is 
$\Theta(\max(u,v)^{2 \omega(d)})$, where $\omega(d)$ is the
number of distinct prime divisors of $d$. In particular this
double sequence is not $q$-automatic for any $q \geq 2$ if $q$ is not a prime power. (This result was generalized to
linear cellular automata by Berth\'e \cite{Berthe}.)

\item If ${\mathbf t} = (t_n)_{n \geq 0}$ is the Thue-Morse 
sequence (defined, e.g., as the fixed point of the morphism 
$0 \to 01$, $1 \to 10$) and $H$ is a polynomial with rational 
coefficients sending the integers to the integers and such 
that $\deg H \geq 2$, then the block complexity of the sequence
$(t_{H(n)})_{n \geq 0}$ grows exponentially
\cite[Corollary~3]{Moshe}, which proves that this sequence 
is not $q$-automatic for any $q \geq 2$ (this was proved only 
for $q = 2$ in \cite{allouche82}). Furthermore it is proved 
in \cite{DMR} that the sequence $(t_{n^2})_{n \geq 0}$ is
normal, and, more generally, in \cite{Mullner} that the sequences
$(u_{n^2} \bmod d)_{n \geq 0}$ are normal, where 
$(u_n)_{n \geq 0}$ is a digital sequence in base $q$ in the 
sense of \cite{Cateland} and $d$ an integer prime to $q-1$
and to $\gcd\{u_n, \ n \in {\mathbb N}\}$.

\item Another example is \cite{Firicel1, Firicel2}, where it is
proved that the fixed point of the morphism $a \to aab$, 
$b \to b$ is not $q$-automatic for any $q \geq 2$, by proving
that its block-complexity is in $\Theta(n^2)$. 

\item The following result was proved in \cite{Hu}: 
if ${\mathbf u} = (u_n)_{n \geq 1}$ is a completely 
multiplicative sequence (i.e., $u_{mn} = u_m u_n$ 
for all $m, n \geq 1$), taking finitely many values in 
a field $K$, and if the number of primes $p$ such that 
$u(p) \neq 1_K$ is finite, then the subword complexity 
of ${\mathbf u}$ is $\Theta(n^t)$ where $t$ is the
number of primes $p$ such that $u(p) \notin \{0_K, 1_K\}$.
An example of application is that the sequence
$((-1)^{\nu_2(n) + \nu_3(n)})_{n \geq 1}$ is not $q$-automatic
for any $q$, where $\nu_p(n)$ is the highest exponent $j$ such
that $p^j \, \mid \, n$: its complexity is in $\Theta(n^2)$.

\item The last two, spectacular, examples that we give in this
section are results from \cite{ABL, AB} and \cite{Bugeaud}: 

\begin{itemize}
\item \cite{ABL, AB} The sequence of digits of an algebraic 
irrational real in base $b \geq 2$ is not $q$-automatic for 
any $q \geq 2$. More generally the main result of 
\cite{ABL, AB} reads: 
{\it The complexity of the $b$-ary expansion of every 
irrational algebraic number satisfies the property
$\liminf_{n \to \infty} p(n)/n = + \infty$.} 

\item \cite{Bugeaud} (also see \cite{AB2}) Let $\alpha$ be a 
positive real which is algebraic over ${\mathbb Q}$ of degree 
at least $3$, and such that its continued fraction expansion 
has finitely many distinct partial quotients. Then this 
sequence of partial quotients is not $q$-automatic for any 
$q \geq 2$. More generally the main result of \cite{Bugeaud} 
reads: {\it If the sequence of partial quotients of a positive 
algebraic real of degree $\geq 3$ takes finitely many distinct 
values, then the (block-)complexity of this sequence of partial
quotients satisfies the property 
$\liminf_{n \to \infty} p(n)/n = + \infty$.} 

\end{itemize}

\end{itemize}

\end{example}

\section{Gaps and runs}\label{gaps-runs}

Cobham \cite{Cobham72} proved the following useful result about 
gaps in automatic sequences. (A similar result was found 
independently by Minsky and Papert \cite{Minsky-Papert}.)
\begin{theorem}
Let ${\bf x} = (x(n))_{n \geq 0}$ be a $k$-automatic sequence 
over $\Delta$. Let $d \in \Delta$. Define $\alpha_j$ to be the 
position of the $j$'th occurrence of $d$ in ${\mathbf x}$. 
(More formally, if $|{\bf x}[0..t-1]|_d = j-1$ and 
${\bf x}[t] = d$, then $\alpha_j = t$.) 
Then either
$$
\limsup_{n \rightarrow \infty} 
{{|{\bf x}[0..n-1]|_d} \over {\log n}} < \infty
\ \ \ \text{or} \ \ \
\liminf_{j \rightarrow \infty} \alpha_{j+1} - \alpha_j < \infty 
\ \ \ \text{(or both)}.
$$
\label{gaps}
\end{theorem}

\begin{remark}
It is possible for both alternatives to hold. For example, 
consider the characteristic sequence of the set
$\{ 2^n \suchthat n \geq 1 \} \, \cup \, \{2^n - 1 \suchthat n \geq 1 \}$,
and $d = 1$.
\end{remark}
As an application let us prove the following:
\begin{corollary}
Let $p$ be a polynomial with rational coefficients such that
$p({\mathbb N}) \subseteq {\mathbb N}$. Then the characteristic 
sequence ${\mathbf c}$ of the set $\{ p(i) \suchthat i \geq 0 \}$ 
is $k$-automatic if and only if $\deg p < 2$.
\end{corollary}

\begin{proof}
If $\deg p < 2$, then this characteristic sequence is ultimately
periodic, and hence $k$-automatic. Otherwise assume $\deg p \geq 2$,
and $\bf c$ is $k$-automatic. Take $d = 1$ in Theorem~\ref{gaps}.  
We have $\alpha_{j+1} - \alpha_j = p(j'+1) - p(j')$ for 
$j' = j + c$, and $j$ sufficiently large and $c$ a constant. But 
this difference is a polynomial of degree $(\deg p) - 1$ and 
hence goes to $\infty$ as $j$ gets large, so the theorem tells 
us that
$$
\limsup_{n \rightarrow \infty} 
{{|{\bf x}[0..n-1]|_d} \over {\log n}} < \infty .
$$
But $|{\bf x}[0..n-1]|_d = \Theta(n^{1/s})$, where $s = \deg p$,
a contradiction.   Hence $\bf c$ cannot be $k$-automatic.
\end{proof}

\bigskip

\noindent
Let us give another application found in \cite{KLR}, which in
particular answers a question in \cite{ASS}.

\begin{corollary}[K\"arki--Lacroix--Rigo]
Let $r$ be an integer $\geq 1$. Let $F$ be the set of maps
$\{\varphi_0, \varphi_1, \ldots, \varphi_r\}$ with 
$\varphi_0(x) = x$ and there exist $k_i, \ell_i \in {\mathbb Z}$, 
with $2 \leq k_1 \leq k_2 \leq \ldots k_r$ such that 
$\varphi_i(x) = k_i x + \ell_i$ for each $i \in [1, r]$.
Let $F(S)$ be defined for any set of integers $S$ by 
$F(s) := \{\varphi(s), s \in S, \varphi \in F \}$.
Let $I$ be any finite set of integers. Define $F^0(I) = I$, and
$F^{m+1}(I) := F(F^m(I))$ for $m \geq 0$. 
Let $X := \cup_{m \geq 0} F^m(I)$ and suppose that 
$X \subset {\mathbb N}$. If $\sum_{1 \leq t \leq r} k_t^{-1} < 1$
and if there exist $i, j$ such that $k_i$ and $k_j$ are 
multiplicatively independent, then the characteristic sequence
of $X$ is not $k$-automatic for any $k \geq 2$.
\end{corollary}

Another result about gaps is the following theorem 
\cite{Minsky-Papert}.

\begin{theorem}[Minsky-Papert] \label{minsky-papert}
Let ${\mathbf a}$ be a $k$-automatic sequence. Let $a$ be a value
occurring in ${\mathbf a}$ infinitely often. Suppose that the
frequency of occurrences of $a$ in ${\mathbf a}$ is zero. Then,
letting $\alpha_j$ denote the index of $j$-th occurrence of $a$,
one has $\limsup \alpha_{j+1} / \alpha_j > 1$.
\end{theorem}

As indicated above (second item in Example~\ref{ex-kernel2}) 
an application of this result can be found, e.g., in \cite{Yazdani}.
Note that a more precise version of Theorem~\ref{minsky-papert}
was given by Cobham in \cite[Theorem\ 12]{Cobham72}: it was used 
to prove results of non-automaticity, e.g., in 
\cite{Allouche-Thakur}.

\section{Dirichlet series}

Given an automatic sequence $(a_n)_{n \geq 0}$ with values in
the complex numbers, it is proved in \cite{AMFP} that the 
Dirichlet series $\sum a_n/(n+1)^s$ possesses a meromorphic 
continuation to the entire complex plane, and that its poles, 
if any, belong to a finite number of left semi-lattices.
This result was used in \cite{Coons} to prove that certain
arithmetic sequences are not automatic: their Dirichlet series
cannot be meromorphically continued to the whole complex plane,
or the continuation violates the condition on the poles given 
above.

\begin{example}\label{ex-coons}
Several arithmetic sequences were proved to be non-automatic 
\cite{Coons} by using the properties of their Dirichlet 
series.
In particular, we mention the following:
\begin{itemize}

\item Let $\Omega(n)$ be the number of primes (counted with 
multiplicity) that divide $n$. The Liouville function $\lambda$
is defined by $\lambda(n) = (-1)^{\Omega(n)}$. Then the 
sequence $(\lambda(n))_{n \geq 1}$ is not $q$-automatic for 
any $q \geq 2$.

\item The characteristic function of the prime numbers is not
$q$-automatic for any $q \geq 2$. (Note that this was proved
in a different way in \cite{Hartmanis-Shank}.)

\item The characteristic function of the prime powers is not
$q$-automatic for any $q \geq 2$. (Note that this was proved
in a different way in \cite{Minsky-Papert}.)

\item The sequences $(q_m(n))_{n \geq 1}$, for $m \geq 2$, 
are not $q$-automatic for any $q \geq 2$, where $q_m(n)$ 
is defined by
$$
q_m(n) =
\begin{cases}
0, &\text{if $p^m \, \mid\, n$ for some prime $p$;} \\
1, &\text{otherwise.}
\end{cases}
$$
\end{itemize}

\end{example}

\begin{remark}
Also note that Dirichlet series were used in \cite{Hu} 
to give an alternative proof of the non-automaticity of 
certain sequences.
\end{remark}

\section{Orbit properties}\label{orbit properties}

We define the {\it shift map\/} $T$ on sequences taking their values 
in a finite set as follows.

\begin{definition}

\ { }

\begin{itemize}

\item 
If ${\mathbf u} = (u_n)_{n \geq 0}$ is a sequence, then
$T{\mathbf u} := (v_n)_{n \geq 0}$, where $v_n := u_{n+1}$ 
for all $n \geq 0$. 
In other words, $T(u_0 u_1 u_2 \ldots) := u_1 u_2 u_3 \ldots$.

\item
The orbit of a sequence ${\mathbf u} = (u_n)_{n \geq 0}$ under the
shift is the set of sequences obtained from ${\mathbf u}$ by
iterating $T$, namely $\{{\mathbf u}, T({\mathbf u}), 
T^2({\mathbf u}), \ldots \} = \{T^k({\mathbf u}), \ k \geq 0\}$.

\item
The orbit closure of a sequence is the closure (for the product 
topology on the set of sequences) of the closure of this orbit.
\end{itemize}
\end{definition}

As proved in \cite[Theorem~6]{ARS}: {\it Let $q \geq 2$. The 
lexicographically least sequence in the orbit closure of a 
$q$-automatic sequence is also $q$-automatic.} Thus we get the following result:
\begin{theorem}[{\rm \cite{ADQ}}]\label{subalphabet}
Let ${\mathbf x} = (x_n)_{n \geq 0}$ be a sequence over some 
alphabet ${\cal A}$. Let ${\cal A'}$ be a proper subset of 
${\cal A}$. Suppose that there exists a sequence 
${\mathbf y} = (y_n)_{n \geq 0}$ on ${\cal A}'$ with 
the property that each of its prefixes appears in 
${\mathbf x}$.
Let $d \geq 2$. If no sequence in the closed orbit of 
${\mathbf y}$ under the shift is $q$-automatic, then 
${\mathbf x}$ is not $q$-automatic.
In particular, let ${\mathbf x} = (x_n)_{n \geq 0}$ be a 
sequence over some alphabet ${\cal A}$. Let ${\cal A'}$ be a 
proper subset of ${\cal A}$. Suppose that there exists a 
sequence ${\mathbf y} = (y_n)_{n \geq 1}$ on ${\cal A}'$ 
with the property that each of its prefixes appears in 
${\mathbf x}$. If ${\mathbf y}$ is Sturmian, or  
${\mathbf y}$ is uniformly recurrent and its complexity 
is not in ${\cal O}(n)$, then ${\mathbf x}$ is not 
$q$-automatic for any $q \geq 2$.
\end{theorem}

\begin{example}
Theorem~\ref{subalphabet} was used in \cite{ADQ} to prove 
that the two fixed points of morphisms respectively given in 
\cite[Theorem~2.9]{Benli} and \cite[Theorem~4.5]{Bartholdi} 
are not $q$-automatic for any $q \geq 2$, namely
\begin{itemize}

\item
The fixed point beginning with $a$ of the morphism 
$a \to aca$, $b \to bc$, $c \to b$ is not automatic. 

\item
The fixed point beginning with $a$ of the morphism 
$a \to aca$, $c \to cd$, $d \to c$ is not automatic.

\end{itemize}

\end{example}

\begin{remark}
One of the referees of the paper \cite{ADQ} has just noted that 
Theorem~\ref{subalphabet} can also be deduced from Theorem~A
in \cite{BKK}. Furthermore this approach does not need that 
${\cal A}'$ be a proper subset of ${\cal A}$.
\end{remark}

\section{When non-$q$-automaticity implies non-automaticity}

Recall that a sequence is called non-automatic if it is not 
$q$-automatic for any integer $q \geq 2$. A nice and deep 
theorem in \cite[Theorem~1 and Corollary~6]{Durand} implies 
the following result.

\begin{theorem}[{\rm \cite{Durand}}]\label{durand}
Let $A$ be a finite alphabet. Suppose that $(a_n)_{n \geq 0}$ 
is the image by a non-erasing morphism of an iterative fixed 
point beginning with some letter $b$ of a morphism $\sigma$, 
such that all letters in $A$ occur in $\sigma^{\infty}(b)$. 
Let $\alpha$ be the dominant eigenvalue of the adjacency matrix 
of $\sigma$. If $(a_n)_{n \geq 0}$ is $q$-automatic and not 
ultimately periodic, then $\alpha$ and $q$ must be 
multiplicatively dependent (i.e., there exist two integers 
$s$ and $t$ with $\alpha^s = q^t$). In particular, if $\alpha$ 
is an integer, and the sequence $(a_n)_{n \geq 0}$ is automatic
and not ultimately periodic, then it must be 
$\alpha$-automatic.
\end{theorem}

\begin{example}
Let us consider one more time the morphism $a \to aab$, $b \to b$.  
We know that it is not $2$-automatic (as proved 
in\cite{all-shall-betrema}), while the dominant eigenvalue of 
the adjacency matrix is equal to $2$. This proves once more that 
it is not $q$-automatic for any $q \geq 2$.
\end{example}

\begin{remark}
Theorem~\ref{durand} above contains the celebrated Cobham theorem
\cite{Cobham69} which asserts that, {\it if $q$ and $r$ are two 
integers $\geq 1$ that are multiplicatively independent, then a 
sequence that is both $q$-automatic and $r$-automatic must be 
eventually periodic.} (The dominant eigenvalue of the 
adjacency matrix of the uniform morphism behind a $d$-automatic 
sequence is $d$.) This theorem can be used to prove that a sequence 
is not $r$-automatic, if it is already known to be $q$-automatic 
and not eventually periodic (with $q$ and $r$ multiplicatively 
independent). A nice generalization can be found in \cite{BK2}, 
where the notion of an {\it almost everywhere $q$-automatic} 
sequence (suggested by J.-M. Deshouillers) is introduced (this is 
a sequence that coincides with a $q$-automatic sequence on a set of 
natural density $1$): in particular, the authors of \cite{BK2} use 
this notion to prove a result of non-automaticity.

\end{remark}

\section{A ``dynamical'' approach}\label{dyn}

In this section we will consider {\it (discrete) dynamical systems}
associated with sequences. The papers on the subject usually stick
to the language of dynamics, which can differ from the language
used by combinatorists on words. We will expand a bit the dynamical
approach: in particular, we will try, as far as possible, to 
use the language of combinatorics of words, e.g., by 
``translating'' the dynamical notions. Our purpose is to give
some details about how to use the concept of {\it dynamical 
eigenvalues} for morphic sequences to spot whether a fixed point 
of a morphism $\varphi$ is automatic.

\bigskip

If ${\mathbf a}=(a_n)_{n\geq 0}$ is a sequence on a finite alphabet
$A$, it defines a {\it shift dynamical system $(X_{\mathbf a}, T)$}
as follows. Let $T$ be the left shift map defined in Section \ref{orbit properties} and let $X_{\mathbf a}$ be the 
(topological) closure of $\{T^k({\mathbf a}),  k \geq 0\}$, where 
the topology on sequences is the one where two sequences  are 
close if they agree on a long enough initial block. The elements
of $X_{\mathbf a}$ are exactly the  sequences all of whose subwords
appear in ${\mathbf a}$.  If $\varphi$ is a primitive morphism, 
then any {\it ultimately $\varphi$-periodic ${\mathbf a}$}
i.e., any sequence ${\mathbf a}$ such that 
$\varphi^i({\mathbf a}) = \varphi^j({\mathbf a})$ for some 
$i, j \geq 1$ with $i \neq j$) defines the same shift dynamical 
system, so we can  write $X_{\varphi}$ instead of $X_{\mathbf a}$. 
Also with the assumption of primitivity, 
$X_{\varphi}= X_{\varphi^j}$ for each $j\geq 1$, so we can assume,
up to replacing $\varphi$ by an iterate, that an ultimately
$\varphi$-periodic sequence is a $\varphi$-fixed point.

\bigskip

We say that  $\lambda\in {\mathbb C}$ is a (topological) {\it 
dynamical eigenvalue} for $(X_{\bf a}, T)$  if there is a 
(continuous) function $f:X_{\bf a} \rightarrow \mathbb C$ such 
that $f\circ T=\lambda f$. An eigenvalue for a dynamical system 
captures notions of periodicity in it. For example, suppose that
${\mathbf a}$ is a fixed point of a  primitive $q$-uniform morphism 
$\varphi$. Recognizability \cite{Mosse} implies that if 
${\mathbf a}$ is not ultimately periodic, then any 
$x\in X_{\mathbf a}$ can be desubstituted in a unique way, i.e., 
$x= T^k \varphi^n (y)$ for a unique $y\in X_{\mathbf a}$ and a
unique $k \in [0, q^n)$. Using this desubstitution for $x$, if we 
define $f(x) = e^{\frac{2\pi i k }{ q^n}}$, then $f$ is an 
eigenfunction for the eigenvalue $e^{\frac{2\pi i }{ q^n}}$. 
Therefore  for any natural number $n$, the number $\lambda=e^{\frac{2\pi i}{q^n}}$ 
is an eigenvalue for $(X_{\mathbf a},T)$.

\bigskip

The work of Kamae \cite{Kamae} and Dekking \cite{Dekking} shows 
that for primitive $q$-uniform morphisms, the only other 
possibility for an eigenvalue is some $e^{\frac{2\pi i}{h}}$, 
where $h\in \mathbb N$ is prime to $q$ (and in fact turns out to 
divide $q-1$).   Here $h$ is the {\em height} of $\varphi$, or, equivalently, 
the height of one (any) of its fixed points ${\mathbf a}$. Define 
$w$ to be a {\em return word} to the letter $a$ for $\varphi$ if 
$w$ starts with $a$ and $wa$ occurs in $\varphi^n(b)$
for some $n$ and $b$. The height $h$ is defined to be
$$ 
h= h(\varphi)=h({\mathbf a}):= \gcd \{|w|,  \ w \ 
\text{is a return word to $a_0$ in ${\bf a}$} \}.
$$
As $\varphi$ is primitive, its height does not depend on the choice 
of the fixed  point ${\bf a}$. The height of $\varphi$  entirely 
depends on the return word structure of $\varphi$. 

\bigskip

Call a sequence ${\boldsymbol a}$
{\em aperiodic} if it is not eventually periodic, and call it {\em minimal} if
for any factor ${\boldsymbol w}$ that appears in  ${\boldsymbol a}$, there is a constant $k$ such  that ${\boldsymbol w}$ appears appears in every factor of ${\boldsymbol a}$ of length at least $k$.
The theorem of Dekking and Kamae was extended to shifts defined by  aperiodic
automatic sequences  which are codings of a primitive $q$-morphism in \cite{MY}.
 Combining this  result with the fact that any minimal automatic sequence can be realised as the  pointwise image of a fixed point of a primitive 
$q$-uniform morphism \cite{Cobham72}, we obtain the following. This result can be seen 
as a dynamical version of Cobham's theorem.

\begin{theorem}[M\"ullner--Yassawi]\label{dekking}
Let ${\boldsymbol a}$ be a minimal aperiodic $q$-automatic sequence,
and let $h$ be the height of ${\mathbf a}$. 
Then the eigenvalues of $X_{\mathbf a}$ are the $q^n$-th roots 
of unity, $\forall n\geq 1$, and $e^{\frac{2\pi i}{h}}$.
\end{theorem}

\bigskip

Another way of capturing the notion of an eigenvalue $\lambda$ 
for $\varphi$ is as follows. Given ${\mathbf a}= (a_n)_{n\geq 0}$, 
a fixed point of a $q$-uniform morphism, we notice that if 
$a_j=a_k$, $j<k$, i.e., if  $w=a[j..k-1]$ is a return word to $a_j$,  
then for each natural number $\ell$, 
$a[jq^{\ell}..(j+1)q^{\ell}-1] = a[kq^{\ell}.. (k+1)q^{\ell}-1]$. 
If $n\in {\mathbb N}$, then writing $\lambda = 
e^{\frac{2\pi i}{q^n}}$, we capture the continuity of the
eigenfunction $f$ by noting that 
\begin{equation}  \label{eq:eigenvalue-eqn-constant}    
\lim_{\ell \to \infty} \lambda^{|\varphi^{\ell}(w)|} = 
\lim_{\ell \to \infty} \lambda^{|w| q^{\ell} }= 1.
\end{equation}

In fact this limit is attained for $\ell \geq n$, but Host 
\cite{Host} showed that this notion extends to fixed points 
${\mathbf a}$ for arbitrary primitive morphisms $\varphi$, and 
there in general $\lambda$ is not necessarily a root of unity\footnote{In fact Host proved more, namely that any {\em measurable} eigenvalue, i.e.,  one which has a Borel-measurable eigenfunction, must be a topological eigenvalue, i.e., one which has a continuous eigenfunction, by showing that a measurable eigenvalue necessarily satisfies \eqref{eq:eigenvalue-eqn-3}.}.

\begin{theorem}[Host]\label{thm:Host}
Let  $\varphi$ be a primitive morphism with an aperiodic fixed 
point and  such that whenever $\varphi(a)$ starts with $b$, 
then $\varphi(b)$ starts with $b$.
Then $\lambda$ is an {\it eigenvalue} for $(X_\varphi, T)$ 
if and only if
\begin{align} \label{eq:eigenvalue-eqn-3}
\lim_{\ell\rightarrow \infty}\lambda^{|\varphi^\ell(w)|} =  1
\end{align}
whenever $w$ is a return word.
\end{theorem}

One could extend these results to non-primitive morphisms, 
although care will need to be taken with letters that are not
recurrent, the existence of eventually periodic $\varphi$-fixed 
points, etc.

The following example illustrates Host's result.

\begin{example}
Consider the Fibonacci morphism $a\to ab$, $b\to a$. It can be 
checked that $1$ is the gcd of all return words to $a$ and also 
to $b$, so we can apply Lemma \ref{lem:matrix-trivial-coboundary}. 
The adjacency matrix has unit determinant, leading eigenvalue the 
golden ratio $\phi$, with an eigenvector whose entries are in 
${\mathbb Z}[1,\frac{1}{\phi}]$.
Lemma~\ref{lem:matrix-trivial-coboundary} tells us that the 
eigenvalues of ${\bf a}$ must belong to  
${\mathbb Z}[1,\frac{1}{\phi}]$. In particular, ${\mathbf a}$ 
cannot be $q$-automatic for any $q$.
\end{example}

\medskip

Note that if ${\mathbf a}$ is aperiodic, $q$-automatic and also 
the fixed point of a primitive morphism $\varphi$, then, since 
$X_{\mathbf a}= X_{\varphi}$, the eigenvalues of $X_\varphi$ must 
equal the eigenvalues of $X_{\mathbf a}$, so Theorem~\ref{dekking}
tells us that $\varphi$ must have all $q^{n}$-th roots of unity 
as eigenvalues. From \eqref{eq:eigenvalue-eqn-3} we immediately 
see that $q^n$ must divide $|\varphi^\ell(w)|$ whenever $w$ is a 
return word (for large enough $\ell$). The following lemma, a mild
modification of one in \cite{Host}, tells us how to get
information about the eigenvalues of $(X_\varphi, T)$ from the
eigenvalues of the adjacency matrix $M$ of $\varphi$, in the case 
that it is invertible and the height of $\varphi$ equals 1.
Let ${\mathbf t}$ denote the row vector all of whose entries 
equal $t$.

\begin{lemma}\label{lem:matrix-trivial-coboundary}
Let $\varphi$ be a primitive morphism on the alphabet $A$ of 
cardinality $d$, with an aperiodic fixed point ${\mathbf a}$. 
If $\lambda=e^{2\pi i t}$ satisfies
\begin{align}  \label{eq:matrix-trivial-coboundary} 
\lim_{\ell\rightarrow \infty}\lambda^{|\varphi^\ell(a)|} = 1
\end{align}
for each $a\in A$, then we can write 
${\mathbf t}= { \mathbf t }_1+ {\mathbf t}_2$ where  
${\mathbf t} _1M^n \rightarrow 0$ and
$ {\mathbf t}_2M^n \in {\mathbb Z}^d$ for all $n$ large.
Furthermore  if $M$ is invertible and $\psi$ is any eigenvalue for 
$M$ with $|\psi|\geq 1$ and eigenvector ${\boldsymbol \omega}$, 
then $t= \frac{r}{s}$ where $r$ is an integer combination of the
entries of $ {\boldsymbol \omega}$ and $s$ divides $ \det(M^k)$ 
for some $k$.
\end{lemma}

\proof

Let $A = \{ a_1, \dots,a_d\}$. Recalling  that the entries of the 
$i$-th column of $M^n$ sum to $|\varphi^n(a_i)|$, 
Assumption~\eqref{eq:matrix-trivial-coboundary} implies that
$$
\lim_{n\rightarrow \infty }{\mathbf t} M^n \equiv 
{\bf 0} \ (\bmod \, 1).
$$
We can therefore write
$$
{\mathbf t} M^n = {\mathbf u}_n  +{\mathbf v}_n
$$
where ${\mathbf u}_n\in {\mathbb Z}^d$ and 
${\mathbf v}_n\rightarrow  {\bf 0}$ as $n\rightarrow \infty$.
For each $n$
\begin{align*}
{\mathbf u}_{n+1} + {\mathbf v}_{n+1} 
     &= \mathbf t M^{n+1} = \left(  {\mathbf u}_n 
     +{\mathbf v}_n\right) M \\
     &= {\mathbf u}_n  M +{\mathbf v}_n M,
\end{align*}
so that
$$
{\mathbf u}_{n+1} - {\mathbf u}_n M             
= {\mathbf v}_n M - {\mathbf v}_{n+1}.
$$
As the right hand side of this last expression converges to the 
vector ${\mathbf 0}$, so does the left.
But ${\mathbf u}_{n+1} - {\mathbf u}_n M  $ is an integer 
valued row vector. We conclude that there exists $n^*$ such that  
$$
{\mathbf u}_{n^*+m} -  {\mathbf u}_{n^*}  M^m= {\bf 0}
\text{ for all } m\geq 1.
 $$
 
 We can find a vector $\mathbf t_2$ such that 
 $\mathbf  t_2 M^{n^{*}+d}  = {\mathbf u}_{n^*} M^d$, and so
 $$
 {\mathbf u}_{n} = \mathbf t_2 M^n
 \text{ for } n\geq n^{*}+d+1.
 $$

 Write $\mathbf  t_1 := \mathbf  t - \mathbf t_2$. Then
 $$
 \mathbf t_1 M^n = \mathbf t  M^n - \mathbf t_2 M^n 
 = \mathbf t M^n - {\mathbf u}_n = {\mathbf v}_n
 $$
 for $n\geq n^{*}+d+1$. 
 But ${\mathbf v}_n \rightarrow {\mathbf 0}$ as 
 $n\rightarrow \infty$, so 
 $\mathbf t_1 M^n \rightarrow {\mathbf 0}$ as 
 $n\rightarrow \infty$.
Now since $|\psi|\geq 1$,
we conclude that  $\mathbf t_1 $ is orthogonal to 
${\boldsymbol \omega}$.
We have
$$
t = \left\langle \mathbf t, {\boldsymbol \omega} \right\rangle 
= \langle \mathbf t_1+\mathbf t_2, {\boldsymbol \omega} 
\rangle = \langle \mathbf t_2,{\boldsymbol \omega}\rangle.
$$

Finally, if $M$ is invertible, then 
$\mathbf t_2= {\mathbf u}_{n^{*}+d+1}  M^{-(n^{*}+d+1)}$. 
The result follows. \endpf

\begin{remark}
There is an analogue of Lemma \ref{lem:matrix-trivial-coboundary} 
in the case where the greatest common divisor of all return words  
to some $a$ is larger than $1$, i.e., where the limit
\eqref{eq:matrix-trivial-coboundary} is not constantly equal to $1$.
Namely we apply the same proof as in the lemma, but starting with
${\mathbf t} (M-I)M^n$ instead of ${\mathbf t} M^n$. Also 
note that Ferenczi, Mauduit and Nogueira \cite{FMN} describe how 
to recover all dynamical eigenvalues from the adjacency matrix $M$.
\end{remark}

\begin{example}
Consider the morphism $1 \to 2$, $2\to 211$ with fixed point 
${\mathbf a} = 21122211211\cdots$ Since $11$ and $22$ each appear in 
${\mathbf a}$, then $1$ is the gcd of all return words to $1$ and
also to $2$, so we can apply 
Lemma~\ref{lem:matrix-trivial-coboundary}.
The adjacency matrix has determinant $-2$ and eigenvalues $-1,2$.  
Suppose that $\lambda=e^{2\pi i t}$ is a dynamical eigenvalue 
(for $(X_\varphi, T)$). Lemma~\ref{lem:matrix-trivial-coboundary}
tells us that we can write 
${\mathbf t}= { \mathbf t }_1+ {\mathbf t}_2$ where  
${\mathbf t} _1M^n \rightarrow 0$ and $ {\mathbf t}_2 M^j$ 
belong to ${\mathbb Z}^d$ for some $j\in {\mathbb N}$. But as $M$ 
has no eigenvalues inside the unit circle, this means that 
${\mathbf t }_1=0$, so that 
${\mathbf t }= (\frac{a}{2^j}, \frac{a}{2^j})$ for some odd
integer $a$. Hence  the fixed point cannot be $q$-automatic, for 
$q>2$. Now one verifies that if $n\geq 2$ and $j\geq 1$, then
${\mathbf t} M^n = (\frac{a_n}{2^j}, \frac{b_n}{2^j})$, where 
$a_n$ and $b_n$ are odd. But ${\mathbf t}M^n\rightarrow 0 
\ (\bmod \ {\mathbb Z}^2)$. Therefore $j=0$. Hence {\bf a} does not 
have all (nor any, in fact) $2^{j}$-th roots of unity as dynamical
eigenvalues, and therefore it is not $2$-automatic.
\end{example}

\begin{example}
Consider any morphism with adjacency matrix 
$\begin{pmatrix} 3&6\\1&2 \end{pmatrix}$. Here also, $1$ is the 
gcd of all return words to each of the two letters.
We cannot apply Lemma~\ref{lem:matrix-trivial-coboundary} since 
the adjacency matrix is not invertible. But we have 
$|\varphi^n(a)|=5^{n-1}\cdot4$, $|\varphi^n(b)|=5^{n-1}\cdot8$. 
By Theorem~\ref{thm:Host}, $e^{\frac{2\pi i}{5^n} }$ is an 
eigenvalue for each $n$. This suggests that any fixed point might 
be $5$-automatic, and indeed \cite[Theorem~1]{ADQ} gives this.
\end{example}

\section{Conclusion. A strategy for proving that fixed 
points of non-uniform morphisms are not automatic} 

The previous sections addressed the question whether 
``general'' sequences are automatic, but with an emphasis 
on sequences that are fixed points of non-uniform morphisms. 
For the latter, what precedes suggest a general strategy. 
Suppose that we are given the iterative fixed point 
${\mathbf u} = (u_n)_{n \geq 0}$ of a non-uniform morphism 
$\varphi$, whose transition matrix has spectral radius $\rho$. 
The following steps can be followed.

\bigskip

\noindent
First preliminary case: the morphism $\varphi$ is primitive.

\begin{itemize}

\item If the morphism $\varphi$ is primitive and if $\rho$ is 
not an integer, then ${\mathbf u}$ is not $q$-automatic 
for any $q \geq 2$. The case where $\rho$ is irrational
is covered in Theorem~\ref{irrational2} above. The case
where $\rho$ is rational but not integer is addressed at the
end of the first item below.

\item Compute the dynamical eigenvalues of $X_{\mathbf u}$ and
try to apply Theorem~\ref{dekking} in Section~\ref{dyn} above
\end{itemize}

\noindent
General case: no assumption of primitivity for the morphism
$\varphi$.

\begin{itemize}

\item If $\rho^k$ is never an integer for $k$ integer 
$\geq 1$, then the sequence $(u_n)_{n \geq 0}$ is not 
$q$-automatic for any $q \geq 2$ (Theorem~\ref{durand} above).
Note that in particular, this is the case if there exists 
some integer $t \geq 1$ such that $\rho^t$ is a rational number but not an integer.

\item If $\rho^k = d$ for some integer $d \geq 2$, thus, using 
Theorem~\ref{durand} again, $(u_n)_{n \geq 0}$ is either not 
$q$-automatic for any $q \geq 2$, or it is $d^{\ell}$-automatic
for some $\ell \geq 1$ (hence $d$-automatic), or it is 
ultimately periodic (hence $d$-automatic). Thus, proving that 
the sequence $(u_n)_{n \geq 0}$ is not $q$-automatic for any 
$q \geq 2$ is the same as proving that it is not $d$-automatic.

\end{itemize}

Thus we see that, up to replacing the morphism $\varphi$ 
with some integer power $\varphi^k$ (note that this replaces 
$\rho$ with $\rho^k$), the case that is not ``immediate'' now
is the case where $\rho$ is an integer. In this situation, we
can (try to) use one of the properties previously described: 

\begin{itemize}

\item[$*$] exhibiting infinitely many distinct elements
in the $d$-kernel of $(u_n)_{n \geq 0}$, 

\item[$*$] proving that some block occurs in $(u_n)_{n \geq 0}$
with irrational frequency, 

\item[$*$] proving that the complexity of $(u_n)_{n \geq 0}$ 
is not in ${\mathcal O}(n)$, 

\item[$*$] finding gaps of ``wrong'' size in the sequence of
integers $\{n, \ u_n = a\}$ for some value $a$, 

\item[$*$] looking at the Dirichlet series associated with 
$(u_n)_{n \geq 0}$, 

\item[$*$] studying the closed orbit of $(u_n)_{n \geq 0}$ 
under the shift, and so forth.

\end{itemize}

\bigskip

\noindent
{\bf Acknowledgments} \ We thank Val\'erie Berth\'e for fruitful 
discussions. We thank Dan Rust for having suggested looking at
return words.


\begin{thebibliography}{99}

\bibitem{AB} B. Adamczewski and Y. Bugeaud, On the complexity of 
algebraic numbers. I. Expansions in integer bases, {\it Ann. 
of Math.} {\bf 165} (2007), 547--565.

\bibitem{AB2} B. Adamczewski and Y. Bugeaud, On the complexity of 
algebraic numbers, II. Continued fractions, {\it Acta Math.} 
{\bf 195} (2005), 1--20.

\bibitem{ABL} B. Adamczewski, Y. Bugeaud,  and F. Luca, Sur la 
complexit\'e des nombres alg\'ebriques, {\it C. R. Math. 
Acad. Sci. Paris} {\bf 339} (2004), 11--14.

\bibitem{AAS} G. Allouche, J.-P. Allouche, and J. Shallit, 
Kolam indiens, dessins sur le sable aux \^{\i}les Vanuatu, 
courbe de Sierpi\'nski et morphismes de mono\"{\i}de, 
{\it Ann. Inst. Fourier} {\bf 56} (2006), 2115--2130.

\bibitem{allouche82} J.-P. Allouche, Somme des chiffres et 
transcendance, {\it Bull. Soc. Math. France\,} {\bf 110} 
(1982), 279--285.

\bibitem{allouche90} J.-P. Allouche, Sur la transcendance de 
la s\'erie formelle $\Pi$, {\it J. Th\'eor. Nombres Bordeaux\,}
{\bf 2} (1990), 103--117.

\bibitem{allouche-lotharingien} J.-P. Allouche, Finite 
automata and arithmetic, in {\it S\'eminaire Lotharingien 
de Combinatoire (Gerolfingen, 1993)}, Pr\'epubl. Inst. Rech. 
Math. Av., 1993/34, Univ. Louis Pasteur, Strasbourg, 1993, 
p.~1--18. Available at
\url{https://www.emis.de/journals/SLC/opapers/s30allouche.pdf}.

\bibitem{allouche-hanoi} J.-P. Allouche, Note on the cyclic 
towers of Hanoi, {\it Theoret. Comput. Sci.} {\bf 123} (1994) 
3--7.

\bibitem{allouche-gamma} J.-P. Allouche, Transcendence of 
the Carlitz-Goss gamma function at rational arguments, 
{\it J. Number Theory\,} {\bf 60} (1996), 318--328.

\bibitem{allouche-palanga} J.-P. Allouche, 
Note on the transcendence of a generating function, in 
{\it New Trends in Probability and Statistics}, 
Vol.~4 (Palanga, 1996), VSP, Utrecht, 1997, 461--465.

\bibitem{allouche-berthe} J.-P. Allouche and V. Berth\'e,
Triangle de Pascal, complexit\'e et automates, {\it Bull. 
Belg. Math. Soc. Simon Stevin} {\bf 4} (1997), 1--23.

\bibitem{allouche-goldmakher} J.-P. Allouche and L. Goldmakher, 
Mock characters and the Kronecker symbol, {\it J. Number 
Theory\,} {\bf 192} (2018), 356--372.

\bibitem{allouche-salon} J.-P. Allouche and O. Salon, 
Sous-suites polynomiales de certaines suites automatiques,
{\it J. Th\'eorie Nombres Bordeaux\,} {\bf 5} (1993), 111--121.

\bibitem{AS} J.-P. Allouche and J. Shallit,  
{\it Automatic Sequences: Theory, Applications, 
Generalizations}, Cambridge University Press, 2003.

\bibitem{Allouche-Thakur} J.-P. Allouche and D. S. Thakur,
Automata and transcendence of the Tate period in finite
characteristic, {\it Proc. Amer. Math. Soc.} {\bf 127} 
(1999), 1309--1312.

\bibitem{all-shall-betrema} J.-P. Allouche, J. B\'etr\'ema, 
and J. O. Shallit, Sur des points fixes de morphismes d'un 
mono\"{\i}de libre, {\it RAIRO Inform. Th\'eor. Appl.} 
{\bf 23} (1989), 235--249.

\bibitem{ADQ} J.-P. Allouche, F. M. Dekking, and M. Queff\'elec, 
Hidden automatic sequences, Preprint (2020). Available at 
\url{https://arxiv.org/abs/2010.00920}.

\bibitem{AMFP} J.-P. Allouche, M. Mend\`es France, 
and J. Peyri\`ere, Automatic Dirichlet series, {\it J. Number 
Theory\,} {\bf 81} (2000), 359--373.

\bibitem{ARS} J.-P. Allouche, N. Rampersad, and J. Shallit, 
Periodicity, repetitions, and orbits of an automatic sequence, 
{\it Theoret. Comput. Sci.} {\bf 410} (2009), 2795--2803.

\bibitem{ASS} J.-P. Allouche, J. Shallit, and  G. Skordev,
Self-generating sets, integers with missing blocks, and 
substitutions, {\it Discrete Math.} {\bf 292} (2005), 1--15.

\bibitem{AHPPS} J.-P. Allouche, F. von Haeseler, 
H.-O. Peitgen, A. Petersen, and G. Skordev, Automaticity 
of double sequences generated by one-dimensional linear 
cellular automata, {\it Theoret. Comput. Sci.} {\bf 188} 
(1997), 195--209.

\bibitem{ASWWZ} J.-P. Allouche, J. Shallit, Z.-X. Wen, W. Wu, and
J.-M. Zhang, Sum-free sets generated by the period-$k$-folding 
sequences and some Sturmian sequences, {\it Discrete Math.} 
{\bf 343} (2020), 111958.

\bibitem{Bartholdi} L. Bartholdi, Endomorphic presentations 
of branch groups, {\it J. Algebra\,} {\bf 268} (2003), 
419--443.

\bibitem{bs} L. Bartholdi and O. Siegenthaler, The twisted twin 
of the Grigorchuk group, {\it Internat. J. Algebra Comput.} 
{\bf 20} (2010), 465--488.

\bibitem{Bell} J. P. Bell, The upper density of an automatic 
set is rational, {\it J. Théor. Nombres Bordeaux\,} {\bf32} 
(2020), 585--604.

\bibitem{Benli} M. G. Benli, Profinite completion 
of Grigorchuk's group is not finitely presented,
{\it Internat. J. Algebra Comput.} {\bf 22} (2012), 1250045.

\bibitem{Berthe93} V. Berth\'e, Fonction $\zeta$ de Carlitz 
et automates, {\it J. Th\'eor. Nombres Bordeaux\,} {\bf 5} 
(1993), 53--77.

\bibitem{Berthe94} V. Berth\'e, Automates et valeurs de 
transcendance du logarithme de Carlitz, {\it Acta Arith.} 
{\bf 66} (1994), 369--390.

\bibitem{Berthe95} V. Berth\'e, Combinaisons lin\'eaires 
de $\zeta(s)/\Pi^s$ sur ${\mathbb F}_q(x)$, pour 
$1 \leq s \leq q-2$, {\it J. Number Theory\,} {\bf 53} 
(1995), 272--299.

\bibitem{Berthe} V. Berth\'e, Complexit\'e et automates 
cellulaires lin\'eaires, RAIRO Inform. Th\'eor. Appl {\bf 34} 
(2000), 403--423.

\bibitem{BDDLP} V. Berth\'e, F. Dolce, F. Durand, J. Leroy, and 
D. Perrin, Rigidity and substitutive dendric words, 
{\it Internat. J. Found. Comput. Sci.} {\bf 29} (2018), 
705--720.

\bibitem{BlondinMasse&Brlek&Garon&Labbe:2008}
A.~{Blondin Mass\'e}, S.~Brlek, A.~Garon, and S.~{Labb\'e},
Combinatorial properties of $f$-palindromes in the {Thue-Morse}
  sequence,
{\em Pure Math. Appl.} {\bf 19}(2-3) (2008), 39--52.

\bibitem{Brimkov-Barneva} V. E. Brimkov and R. P.  Barneva, 
Plane digitization and related combinatorial problems, 
{\it Discrete Appl. Math.} {\bf 147} (2005), 169--186.

\bibitem{Bugeaud} Y. Bugeaud, Automatic continued fractions 
are transcendental or quadratic, {\it Ann. Scient. \'Ec. 
Norm. Sup.} {\bf 46} (2013), 1005--1022.

\bibitem{Bugeaud&Kim:2018}
Y.~Bugeaud and D.~H. Kim,
A new complexity function, repetitions in {Sturmian} words, and
  irrationality exponents of {Sturmian} numbers,
{\em Trans. Amer. Math. Soc.} {\bf 371} (2019), 3281--3308.

\bibitem{BK1} J. Byszewski and J. Konieczny, Factors of 
generalised polynomials and automatic sequences, 
{\it Indag. Math. (N.S.)} {\bf 29} (2018), 981--985.

\bibitem{BK2} J. Byszewski and J. Konieczny, A density version of 
Cobham’s theorem, {\it Acta Arith.} {\bf 192} (2020), 235--247.

\bibitem{BKK} J. Byszewski, J. Konieczny, and E. Krawczyk,
Substitutive systems and a finitary version of Cobham’s theorem,
(2021), to appear in {\it Combinatorica}. Available at 
\url{https://arxiv.org/abs/1908.11244}.

\bibitem{Cadic1} C. Cadic, {\it Interpr\'etation 
$p$-automatique des groupes formels de Lubin-Tate et des 
modules de Drinfeld r\'eduits}, Th\`ese, Universit\'e de
Limoges, 1999. Available at
\url{https://tel.archives-ouvertes.fr/tel-00474315/document}.

\bibitem{Cadic2} C. Cadic, Modules de Drinfeld formels et 
alg\'ebricit\'e, {\it C. R. Acad. Sci. Paris, S\'er. I, Math.}
{\bf 327} (1998), 335--338.

\bibitem{Carpi-Maggi} A. Carpi and C. Maggi, On synchronized 
sequences and their separators, {\it RAIRO Inform. Th\'eor. App.}
{\bf 35} (2001), 513--524.

\bibitem{Carpi&DAlonzo:2009}
A.~Carpi and V.~D'Alonzo,
On the repetitivity index of infinite words,
{\em Internat. J. Algebra Comput.} {\bf 19} (2009), 145--158.

\bibitem{Cateland} E. Cateland, {\it Suites digitales et suites
$k$-r\'eguli\`eres}, Th\`ese, Universit\'e Bordeaux 1, 1992. 
Available at 
\url{https://tel.archives-ouvertes.fr/tel-00845511/document}.

\bibitem{Charlier&Rampersad&Shallit:2012}
E.~Charlier, N.~Rampersad, and J.~Shallit,
Enumeration and decidable properties of automatic sequences,
{\em Internat. J. Found. Comp. Sci.} {\bf 23} (2012), 1035--1066.


\bibitem{Christol} G. Christol, Ensembles presque 
p\'eriodiques $k$-reconnaissables, {\it Theoret. Comput. Sci.} 
{\bf 9} (1979), 141--145.

\bibitem{CKMFR} G. Christol, T. Kamae, M. Mend\`es France, and
G. Rauzy, Suites alg\'ebriques, automates et substi\-tutions, 
{\it Bull. Soc. Math. France\,} {\bf 108} (1980), 401--419.

\bibitem{Cobham69} A. Cobham, On the base-dependence of sets of 
numbers recognizable by finite automata, {\it Math. Systems 
Theory\,} {\bf 3} (1969), 186--192.

\bibitem{Cobham72} A. Cobham, Uniform tag sequences, 
{\it Math. Systems Theory} {\bf 6} (1972), 164--192.

\bibitem{Coons} M. Coons, (Non)automaticity of number 
theoretic functions, {\it J. Th\'eor. Nombres Bordeaux\,} 
{\bf 22} (2010), 339--352.

\bibitem{Dekking} F.~M. Dekking, The spectrum of dynamical 
systems arising from substitutions of constant length, 
{\it Z. Wahrscheinlichkeitstheorie und Verw. Gebiete} {\bf 41} 
(1977/78), 221--239.

\bibitem{DMR} M. Drmota, C. Mauduit and J. Rivat, Normality 
along squares, {\it J. Eur. Math. Soc.} {\bf 21} (2019),
507--548.

\bibitem{Durand} F. Durand, Cobham's theorem for 
substitutions, {\it J. Eur. Math. Soc.} {\bf 13} (2011), 
1799--1814.

\bibitem{FMN} S. Ferenczi, C. Mauduit and A. Nogueira,
Substitution dynamical systems: algebraic characterization of
eigenvalues, {\it Ann. Sci. \'Ecole Norm. Sup.} (4) {\bf 29} 
(1996), 519--533.

\bibitem{Firicel1} A. Firicel, Quelques contributions \`a 
l'\'etude des s\'eries formelles \`a coefficients dans un 
corps fini, Th\`ese de Doctorat, Universit\' e Claude 
Bernard-Lyon~1, 2010. Available at 
\url{https://tel.archives-ouvertes.fr/tel-00825191}.

\bibitem{Firicel2} A. Firicel, Subword complexity and Laurent 
series, {\it INTEGERS: Elect. J. of Combin. Number Theory} 
{\bf 11B} (2011), \#A7.

\bibitem{Fogg} N. Pytheas Fogg, {\it Substitutions in 
Dynamics, Arithmetics and Combinatorics}, Lect. Notes 
in Math. {\bf 1794}, Springer-Verlag, Berlin, 2002.

\bibitem{Garel:1997}
E.~Garel,
{S\'eparateurs} dans les mots infinis {engendr\'es} par morphismes,
{\em Theoret. Comput. Sci.} {\bf 180} (1997), 81--113.


\bibitem{Goc&Henshall&Shallit:2013}
D.~Go\v{c}, Dane Henshall, and Jeffrey Shallit,
Automatic theorem-proving in combinatorics on words.
{\em Internat. J. Found. Comp. Sci.} {\bf 24} (2013), 781--798.

\bibitem{Goc} D. Goc, L. Schaeffer and J. Shallit, The subword 
complexity of $k$-automatic sequences is $k$-synchronized, 
in M.-P. B\'eal and O. Carton, editors, DLT 2013, Lecture Notes 
in Comput. Sci. {\bf 907}, Springer-Verlag, 2013, 252--263.

\bibitem{Goldrei} D. C. Goldrei, {\it Classic Set Theory.
For Guided Independent Study}, Chapman \& Hall/CRC Press, 
1996. 

\bibitem{Haeseler} F. von Haeseler, {\it Automatic Sequences}, 
de Gruyter Expositions in Mathematics {\bf 36}, Walter 
de Gruyter \& Co., Berlin, 2003.

\bibitem{Hartmanis-Shank} J. Hartmanis and H. Shank, On the 
recognition of primes by automata, {\it J. Assoc. Comput. 
Mach.} {\bf 15} (1968), 382--389.

\bibitem{Hopcroft}
J.~E. Hopcroft and J.~D. Ullman,
{\em Introduction to Automata Theory, Languages, and Computation},
Addison-Wesley, 1979.

\bibitem{Host} B. Host, Valeurs propres des syst\`emes dynamiques 
d\'efinis par des substitutions de longueur variable, 
{\it Ergodic Theory Dynam. Systems\,} {\bf 6} (1986), 529--540.

\bibitem{Hu} Y. Hu, Subword complexity and non-automaticity 
of certain completely multiplicative functions, 
{\it Adv. in Appl. Math.} {\bf 84} (2017), 73--81.

\bibitem{Hu2018} Y. Hu, Transcendence of $L(1,\chi_s)/\Pi$ and
automata, {\it J. Number Theory} {\bf 187} (2018), 215--232.

\bibitem{Kamae} T. Kamae, A topological invariant of 
substitution minimal sets, {\it J. Math. Soc. Japan\,} {\bf 24}
(1972), 285--306.
  
\bibitem{KLR} T. K\"arki, A. Lacroix and M. Rigo, 
On the recognizability of self-generating sets, 
{\it J. Integer Seq.} {\bf 13} (2010), Article 10.2.2.

\bibitem{KK1} O. Klurman and P. Kurlberg, A note on 
multiplicative automatic sequences, {\it C. R. Math. Acad. 
Sci. Paris\,} {\bf 357} (2019), 752--755.

\bibitem{KK2} O. Klurman, P. Kurlberg, A note on 
multiplicative automatic sequences, II, {\it Bull. Lond. 
Math. Soc.} {\bf 52} (2020), 185--188.

\bibitem{Konieczny} J. Konieczny, On multiplicative automatic 
sequences, {\it Bull. London Math. Soc.} {\bf 52} (2020), 
175--184.

\bibitem{KLM} J. Konieczny, M. Lema\'nczyk and C. M\"uller, 
Multiplicative automatic sequences, Preprint (2020). Available
at \url{https://arxiv.org/abs/2004.04920}.

\bibitem{Li} S. Li, On completely multiplicative automatic 
sequences, {\it J. Number Theory\,} {\bf 213} (2020), 388--399.

\bibitem{Lipka} E. Lipka, Automaticity of the sequence of the last 
nonzero digits of $n!$ in a fixed base, {\it J. Th\'eor. Nombres 
Bordeaux\,} {\bf 31} (2019), 283--291.

\bibitem{Mendes-Yao} M. Mend\`es France and J.-Y. Yao, 
Transcendence and the Carlitz-Goss gamma function, 
{\it J. Number Theory\,} {\bf 63} (1997), 396--402.

\bibitem{Mignosi&Restivo:2012}
F.~Mignosi and A.~Restivo,
Characteristic {Sturmian} words are extremal for the {Critical
  Factorization Theorem},
\newblock {\em Theoret. Comput. Sci.} {\bf 454} (2012), 199--205.

\bibitem{Minsky-Papert} M. Minsky and S. Papert, Unrecognizable 
sets of numbers, {\it J. Assoc. Comput. Mach.} {\bf 13} 
(1966), 281--286.

\bibitem{MRSS} L. Mol, N. Rampersad, J. Shallit, and M. Stipulanti,
Cobham’s theorem and automaticity, {\it Internat. J. Found. 
Comput. Sci.} {\bf 30} (2019), 1363--1379.

\bibitem{Moshe} Y. Moshe, On the subword complexity of 
Thue–Morse polynomial extractions, {\it Theoret. Comput. Sci.} 
{\bf 389} (2007), 318--329.

\bibitem{Mosse} B.~Moss\'e, Puissances de mots et 
reconnaissabilit\'e des points fixes d'une substitution, 
{\it Theoret. Comput. Sci.} {\bf 99} (1992), 327--334.

\bibitem{Mullner} C. M\"ullner, The Rudin–Shapiro sequence and 
similar sequences are normal along squares, {\it Canad. J. Math.} 
{\bf 70} (2018), 1096--1129.

\bibitem{MY} C. M\"ullner and R. Yassawi, Automorphisms of 
automatic shifts, {\it Ergodic Theory Dynam. Systems} {\bf 41}
(2021), 1530--1559.

\bibitem{Queffelec} M. Queff\'elec, 
{\it Substitution Dynamical Systems---Spectral Analysis}, 
second edition, Lect. Notes in Math. {\bf 1294}, 
Springer-Verlag, Berlin, 2010.

\bibitem{Schlage-Puchta} J.-C. Schlage-Puchta,
A criterion for non-automaticity of sequences,
{\it J. Integer Seq.} {\bf 6} (2003), Article 03.3.8.
Available at
\url{https://cs.uwaterloo.ca/journals/JIS/VOL6/Puchta/puchta70.html}

\bibitem {Shallit-automaticity4} J. Shallit, Automaticity IV: 
sequences, sets, and diversity, {\it J. Th\'eor. Nombres 
Bordeaux\,} {\bf 8} (1996), 347--367.

\bibitem{Shallit:2021}   
J. Shallit, {\it The Logical Approach To Automatic Sequences: Exploring Combinatorics on Words with \texttt{Walnut}}, in preparation, 2021.

\bibitem{OEIS} N. J. A. Sloane et al.,
\textit{The On-Line Encyclopedia of Integer Sequences}, 2021.
Available at \url{https://oeis.org}.

\bibitem{Spiegelhofer} L. Spiegelhofer, Gaps in the Thue–Morse word,
Preprint (2021). Available at \url{https://arxiv.org/abs/2102.01018}.

\bibitem{Tapsoba} T. Tapsoba, Minimum complexity of automatic 
non-Sturmian sequences, {\it  RAIRO Inform. Th\'eor. Appl.} 
{\bf 29} (1995), 285--291.

\bibitem{Wen-Yao} Z.-Y. Wen and J.-Y. Yao, Transcendence, automata
theory and gamma functions for polynomial rings, {\it Acta 
Arith.} {\bf 101} (2002), 39--51.

\bibitem{Yao} J.-Y. Yao, Crit\`eres de non-automaticit\'e 
et leurs applications, {\it Acta Arith.} {\bf 80} (1997), 
237--248.

\bibitem{Yao2002} J.-Y. Yao, Some transcendental functions 
over function fields with positive characteristic, 
{\it C. R. Math. Acad. Sci. Paris} {\bf 334} (2002), 939--943. 

\bibitem{Yazdani} S. Yazdani, Multiplicative functions and 
$k$-automatic sequences, {\it J. Th\'eor. Nombres Bordeaux\,} 
{\bf 13} (2001), 651--658.

\end{thebibliography}
\end{document}